\documentclass[12pt,reqno]{amsart}
\usepackage{amsmath,amssymb,amsfonts,amsthm}
\usepackage[left=3cm, right=3cm, top=2.8cm, bottom=2.8cm]{geometry}
\usepackage[utf8]{inputenc}
\usepackage[T1]{fontenc}
\usepackage{xcolor}
\usepackage{graphicx}
\usepackage{hyperref}
\hypersetup{
    colorlinks=true,
    linkcolor=blue,
    filecolor=magenta,
    urlcolor=cyan,
}

\newtheorem{theorem}{Theorem}
\newtheorem{lemma}[theorem]{Lemma}
\newtheorem{proposition}[theorem]{Proposition}
\newtheorem{corollary}[theorem]{Corollary}

\theoremstyle{plain}

\begin{document}

\title[There and Back Again]{There and Back Again: Revisiting the Failure of Concatenation in Mittag-Leffler Functions}


\author[P. M. Carvalho-Neto]{Paulo M. Carvalho-Neto}
\address[Paulo M. Carvalho-Neto]{Department of Mathematics, Federal University of Santa Catarina, Florian\'{o}polis - SC, Brazil\vspace*{0.3cm}}
\email[]{paulo.carvalho@ufsc.br}

\author[C. E. T. Ledesma]{Cesar E. T. Ledesma}
\address[Cesar E. T. Ledesma]{Department of Mathematics, Institute for Research in Mathematics, Faculty of Physical and Mathematical Sciences, National University of Trujillo - La Libertad, Peru\vspace*{0.3cm}}
\email[]{ctl\_576@yahoo.es}






\begin{abstract}
The function $t \mapsto E_{\alpha}(\lambda t^\alpha)$ is widely regarded as the fractional analogue of the exponential function, yet its algebraic properties remain poorly understood. In particular, standard references lack a rigorous proof of the failure of the semigroup property. In this work, we fill this gap by providing an analytical proof that the semigroup property holds if and only if either $\alpha = 1$ or $\lambda = 0$. Our result establishes a precise criterion for when Mittag-Leffler dynamics are compatible with semigroup evolution, thereby emphasizing the intrinsic nonlocality of fractional-order differential equations.
\end{abstract}

\maketitle

\section{Origins and Structural Properties}

Introduced in the early 20th century, the Mittag--Leffler function has grown from a theoretical construct in the classification of entire functions into a central object in modern fractional calculus. For $\alpha > 0$, it is defined by the series
\[
E_{\alpha}(z) = \sum_{n=0}^\infty \frac{z^n}{\Gamma(\alpha n + 1)}, \quad z \in \mathbb{C},
\]
and generalizes the exponential function, which is recovered when $\alpha = 1$.

Originally studied by G.~M.~Mittag-Leffler in 1903 in the context of entire functions with controlled growth~\cite{ML01,ML02}, the function $E_\alpha$ gained further mathematical depth through the work of Wiman~\cite{Wi01}, who investigated the asymptotic distribution of its zeros in the complex plane. His results showed that for $\alpha > 2$ the zeros lie on the negative real axis, while for $0 < \alpha < 2$ they spread along logarithmic spirals, a pattern that reflects the complex nature of fractional dynamics.

In recent decades, the Mittag--Leffler function has emerged as the canonical solution of fractional differential equations. More specifically, the solution to the initial value problem
\[
\left\{
\begin{aligned}
    cD_t^\alpha u(t) &= \lambda u(t), \\
    u(0) &= u_0,
\end{aligned}
\right.
\]
with $0 < \alpha < 1$ and $cD_t^\alpha$ denoting the Caputo fractional derivative, is given by $u(t) = E_\alpha(\lambda t^\alpha) u_0$. This expression is often viewed as the natural fractional analogue of the exponential evolution $e^{\lambda t} u_0$ from classical analysis; see \cite[Section 1.8 and Chapter 2]{KiSrTr01} or \cite[Section 1.2 and Chapter 3]{Po01} as two classical references on this subject.

However, despite this structural analogy, the Mittag--Leffler-based solution does not inherit one of the most fundamental features of the exponential: the semigroup (or concatenation) property. In classical settings, the exponential satisfies $e^{\lambda(t+s)} = e^{\lambda t} e^{\lambda s}$, for every pair $t,s\in\mathbb{R}$, yielding a semigroup structure essential to the theory of evolution equations. But for $\alpha \ne 1$ or $\lambda\not=0$, the function $t \mapsto E_\alpha(\lambda t^\alpha)$ fails to satisfy any such property, and no rigorous proof of this fact appears to be available in the literature.

To the best of our knowledge, the first critical examination of the identity
$$E_\alpha\big(\lambda(t + s)^\alpha\big) = E_\alpha\big(\lambda t^\alpha\big)\, E_\alpha\big(\lambda s^\alpha\big), \quad t, s \geq 0,$$
was carried out by Peng and Liu in \cite{PeLi01}. They pointed out that, despite its frequent use in the literature without justification, this multiplicative property is generally invalid, holding only in the particular cases $\alpha = 1$ or $\lambda = 0$. Although the authors mention that this limitation can be heuristically verified via Laplace transform techniques, a complete and rigorous proof was not provided. Their contribution also includes the derivation of an alternative identity satisfied by $E_\alpha(at^\alpha)$, which recovers the multiplicative form above in the classical limit $\alpha \to 1$.

In the present work, we close this gap by rigorously establishing that the semigroup property holds for $E_{\alpha}(\lambda t^\alpha)$ if and only if $\alpha = 1$ or $\lambda = 0$, thereby recovering the classical exponential case. Our result provides a sharp analytical criterion and highlights the structural obstruction to semigroup behavior in fractional-order dynamics, reinforcing the inherently nonlocal and memory-dependent character of such systems.

For a broader overview of Mittag--Leffler functions, including their analytic properties and a wide range of applications, we refer the reader to the comprehensive exposition in~\cite[Chapter~3]{GoKiMaRo1}.

\section{Absence of Semigroup Structure}

We begin by presenting a classical identity involving the Mittag--Leffler functions, which can be found in \cite[Section 5]{HaMaSa1}, namely:

\begin{proposition}\label{202508041600}
Let $\alpha > 0$. Then
$$
\dfrac{d}{dt}\Big(E_{\alpha}(\lambda t^\alpha)\Big) =  t^{\alpha-1} \lambda E_{\alpha,\alpha}(\lambda t^\alpha),
$$
for every $t > 0$. Here, $E_{\alpha,\alpha}(z)$ denotes one of the two-parameter Mittag--Leffler functions, which are also analytic and share similar properties, such as $E_{\alpha,\alpha}(0) = 1$ and $E_{\alpha,\alpha}(t) > 0$ for every $t \geq 0$. For more details on this function, see \cite[Chapter 4]{GoKiMaRo1}.
\end{proposition}

We also recall a classical result from \cite[Theorem 1 in Section 2.1.1]{Ac1}, which characterizes the solutions of a well-known functional equation.
\begin{theorem}\label{202508041539}
Suppose that $f : \mathbb{R} \to \mathbb{R}$ is a continuous function satisfying
$$
f(t+s) = f(t)f(s), \quad \text{for all } t, s \in \mathbb{R},
$$
and $f(0) = 1$. Then there exists a constant $\omega \in \mathbb{R}$ such that
$$
f(\tau) = e^{\omega \tau}, \quad \text{for all } \tau \in \mathbb{R}.
$$
\end{theorem}

The remainder of this section consists of one auxiliary lemma, our main result, and a clearer corollary that better captures the essence of the result we aim to establish.
\begin{lemma}\label{202508041515}
Let $f : [0, \infty) \to \mathbb{R}$ be a continuous function such that $f(\tau) > 0$ for every $\tau \in [0, \infty)$. For each $n \in \mathbb{N}$, define the function $\psi_n : [-n, \infty) \to \mathbb{R}$ by
$$
\psi_n(\tau) = \frac{f(\tau + n)}{f(n)}.
$$
If $f(t + s) = f(t)f(s)$ for all $t, s \in [0, \infty)$, then $\psi_n$ is continuous on $[-n, \infty)$ and coincides with $f$ on $[0, \infty)$. Moreover, if $n > m$, then $\psi_n(\tau) = \psi_m(\tau)$ for all $\tau \in [-m, \infty)$; in particular, $\psi_n$ is a continuous extension of $\psi_m$ on $[-n, \infty)$.
\end{lemma}
\begin{proof}
Let $n \in \mathbb{N}$. The continuity of $\psi_n$ follows directly from the continuity of $f$. Moreover, for every $\tau \in [0, \infty)$, we have
$$\psi_n(\tau) = \frac{f(\tau + n)}{f(n)} = \frac{f(\tau)f(n)}{f(n)} = f(\tau),$$
therefore $\psi_n$ and $f$ coincide on $[0, \infty)$.

Now, suppose $n > m$. For any $\tau \in [-m, \infty)$, we can write
\begin{multline*}\psi_n(\tau)=\frac{f(\tau + n)}{f(n)}=\frac{f(\tau + m + n - m)f(m)}{f(n)f(m)}\\=\frac{f(\tau + m)f(n - m)f(m)}{f(n)f(m)}=\frac{f(\tau + m)}{f(m)}=\psi_m(\tau),\end{multline*}
as desired.
\end{proof}

\begin{theorem}\label{202508041619} For any $\alpha \in (0,1)$ and $\lambda \in \mathbb{R}$, consider the function $f: [0,\infty) \rightarrow \mathbb{R}$ defined by
$
f(t) = E_{\alpha}(\lambda t^{\alpha}).
$
Then $f(t + s) = f(t)f(s)$ for all $t, s \in [0,\infty)$ if and only if $\lambda = 0$.

\end{theorem}

\begin{proof} Let $\alpha \in (0,1)$ be given. If $\lambda = 0$, then $f(t) = 1$ and the identity is trivially satisfied. Conversely, assume for a moment that $f(t + s) = f(t)f(s)$ for all $t, s \in [0,\infty)$. For each $n \in \mathbb{N}$, define $\psi_n : [-n,\infty) \rightarrow \mathbb{R}$ by
$$
\psi_n(t) = \dfrac{f(t+n)}{f(n)}.
$$
Then Lemma \ref{202508041515} ensures that $\psi_n$ is a continuous extension of $f$ over $[-n,\infty)$. Since this property holds for every $n \in \mathbb{N}$, and $\psi_n$ is a continuous extension of $\psi_m$ whenever $n > m$, we may define $F : \mathbb{R} \rightarrow \mathbb{R}$ by
$$
F(t) = \psi_n(t), \quad \text{whenever } t \in [-n,\infty).
$$
This construction guarantees that $F$ is continuous and extends $f$ to the entire real line.

Finally, observe that for any $t, s \in \mathbb{R}$, let $n\in\mathbb{N}$ be such that $t>-n$ and $s>-n$. Then we have
\begin{multline*}F(t+s)=\psi_{2n}(t+s)=\dfrac{f(t+s+2n)}{f(2n)}\\=\left(\dfrac{f(t+n)}{f(n)}\right)\left(\dfrac{f(s+n)}{f(n)}\right)=\psi_n(t)\psi_n(s)=F(t)F(s).\end{multline*}

Since $F(0) = E_\alpha(0) = 1$, we may apply Theorem \ref{202508041539} to deduce that there exists a constant $\omega \in \mathbb{R}$ such that $F(t) = e^{\omega t}$ for every $t \in \mathbb{R}$. Then, by Proposition \ref{202508041600}, we must have
$$
\dfrac{d}{dt}\Big(E_{\alpha}(\lambda t^\alpha)\Big)=\dfrac{d}{dt}\Big(e^{\omega t}\Big)\Longrightarrow t^{\alpha-1} \lambda E_{\alpha,\alpha}(\lambda t^\alpha) = \omega e^{\omega t},
$$
for all $t > 0$, which implies
$$
\lambda = \dfrac{t^{1-\alpha} \, \omega e^{\omega t}}{ E_{\alpha,\alpha}(\lambda t^\alpha)},
$$
for all $t > 0$. Taking the limit as $t \to 0^+$, we conclude directly that $\lambda = 0$, as desired.
\end{proof}

We now present the main conclusion of this work, which we believe complete the discussion on the Mittag--Leffler function and the semigroup (or concatenation) property that it does not satisfy.

\begin{corollary}\label{202508041638}
Let $\alpha \in (0,1]$ and $\lambda \in \mathbb{R}$. Then the functional equation
\begin{equation}
\label{202508041616}
E_{\alpha}(\lambda (t+s)^{\alpha}) = E_{\alpha}(\lambda t^{\alpha})E_{\alpha}(\lambda s^{\alpha}), \quad \text{for all } t, s \in [0, \infty),
\end{equation}
holds if and only if either $\alpha = 1$ or $\lambda = 0$.
\end{corollary}

\begin{proof} If we assume that $\lambda = 0$ or $\alpha = 1$, then $E_{\alpha}(\lambda t^{\alpha}) = 1$ or $E_{\alpha}(\lambda t^{\alpha}) = e^{\lambda t}$, respectively, and therefore \eqref{202508041616} is satisfied; the former case trivially, and the latter as a consequence of the properties of the exponential function. Conversely, if \eqref{202508041616} holds and $\alpha = 1$, there is nothing to prove. Now, if \eqref{202508041616} holds and $\alpha \in (0,1)$, then, thanks to Theorem \ref{202508041619}, we must have $\lambda = 0$, as desired.
\end{proof}

%
%

\section{Concluding Remarks and Extensions}

A natural extension of Corollary~\ref{202508041638} involves matrix-valued functions of the form $E_{\alpha}(A t^\alpha)$, where $A$ is a real square matrix of order $n$. Such generalizations naturally emerge in the study of systems of fractional differential equations and are essential for understanding their structural properties and asymptotic behavior.

However, it is well known that if $A$ is diagonalizable, then $E_{\alpha}(A t^\alpha)$ can be expressed in a basis where it takes the form of a diagonal matrix whose entries are given by $E_{\alpha}(\lambda_j t^\alpha)$, where $\lambda_j$ are the eigenvalues of $A$ (cf. \cite{DaBa01}). Consequently, the result follows as a direct consequence of Corollary~\ref{202508041638}, leading to the following statement:

\begin{corollary}
Let $n \in \mathbb{N}$, $A$ be a real diagonalizable matrix of order $n$, and $\alpha \in (0,1]$. Then the functional equation
\begin{equation}
\label{202508041616}
E_{\alpha}(A (t+s)^{\alpha}) = E_{\alpha}(A t^{\alpha})E_{\alpha}(A s^{\alpha}), \quad \text{for all } t, s \in [0, \infty),
\end{equation}
holds if and only if either $\alpha = 1$ or $A = 0$.
\end{corollary}

A more comprehensive treatment of this property, including the case of non diagonalizable matrices as well as bounded and unbounded linear operators, is currently in preparation and will be presented in a forthcoming work.

\section*{Acknowledgements}

C\'esar E. Torres Ledesma has been supported by CONCYTEC-P, project PE501087741-2024-PROCIENCIA, titled ``An\'alisis Cualitativo de Ecuaciones No Locales y Aplicaciones: Modelos de Difusi\'on An\'omala y Problemas de Obst\'aculo''.

\end{document}